\documentclass[11pt]
{amsart}
\usepackage{amssymb,amsthm,amsmath,currfile}
\usepackage[english]{babel}
\usepackage{verbatim,here}
\usepackage[T1]{fontenc}
\usepackage{graphicx}
\usepackage{epstopdf}
\usepackage{floatflt}
\usepackage{color,xcolor}
\usepackage{enumerate}
\usepackage{bm}
\usepackage{a4wide}

\theoremstyle{plain}
\newtheorem{theor}{Theorem}
\newtheorem{lem}{Lemma}

\theoremstyle{definition}

\usepackage{graphicx}
\usepackage{xcolor}

\begin{document}


\title{Quasiconformal reflection with respect to the boundary of an isosceles trapezoid}

\author{A.~{Kushaeva}}
\address{Kazan Federal University, Kremlyovskaya ul. 35, Tatarstan, 420008 Russia}

\author{K.~{Kushaeva}}
\address{Kazan Federal University, Kremlyovskaya ul. 35, Tatarstan, 420008 Russia}

\author{S.~{Nasyrov}}
\address{Kazan Federal University, Kremlyovskaya ul. 35, Tatarstan, 420008 Russia}



\begin{abstract}
		We establish an upper estimate for the coefficient of quasiconformal reflection with respect to the boundary of an arbitrary isosceles trapezoid in terms of its geometric parameters; the estimate improve the result obtained in the recent paper by S.~Nasyrov, T.~Sugawa and M.~Vuorinen.
\end{abstract}



\maketitle

	\section{Introduction}
	Let $L$ be a Jordan curve on the Riemann sphere $\overline{\mathbb{C}}$, separating it into two parts, $D^+$ and $D^{-}$.  A quasiconformal automorphism $f$  of $\overline{\mathbb{C}}$, reversing orientation, is called a quasiconformal reflection with respect to $L$, if the following hold: $f(D^+)=D^{-}$, $f(D^-)=D^{+}$ and every point $z\in L$ is fixed under the mapping $f$, i.e. $f(z)=z$, $z\in L$ (see, e.g. \cite[ch.~IV]{ahlf}).

It is well known (see, e.g. \cite{ahlf}) that there exists a quasiconformal reflection with respect to $L$ if and only if $L$ is a quasicircle, i.e. an image of a circle under a quasiconformal mapping $h$. If such a quasiconformal mapping  $h$ is $K$-quasiconformal, then $L$ is called a $K$-quasicircle. We note that finding the minimal value of $K$ for a given curve $L$ is a very nontrivial problem even for the case when $L$ is a polygonal line, for example, the boundary of a rectangle (see \cite[p. 455, problem~(20)]{hkv}).

An important and complicated problem is also finding, given a Jordan curve $L$, the minimal value of $K$ such that there exists a $K$-quasiconformal reflection with respecct to $L$;  we will denote such minimal $K$ by $QR_L$. For polygonal lines $L$ this problem was investigated by R.~K\"uhnau (see, e.g. \cite{kuhnau1}). In particular, he showed that if $L$ has an inscribed circle, then $QR_L=2/\alpha-1$ where $\pi\alpha$ is the minimal inner angle of $L$. The exact value of $QR_L$ for a given $L$ can be only obtained for a narrow class of curves (see the survey \cite{kruskal}), therefore, it is of interest to find upper and lower estimates for $QR_L$. In this regard we also note the recent papers \cite{kruskal2,kruskal3,kruskal4}.

For the case when $L$  is the boundary of an arbitrary  rectangle, the exact value of $QR_L$ is unknown yet with the exception of rectangles close to a square. Denote by $\Gamma_m$ the boundary of a rectangle $[0,m]\times [0,1]$, $m\ge 1$. In the paper  \cite{werner} S.~Werner  showed that, for the square ($m=1$) and rectangles close to it ($1<m<1.037$), the value $QR_{\Gamma_m}$ does not depend on $m$ and equals $3$, and at the same time  $QR_{\Gamma_m}>3$ for $m>2.76$. Moreover, in \cite{werner} the following fact is proved.

\begin{theor}\label{w} For every $m\ge 1$ we have
$$
\frac{\pi}{3}\,\,m<QR_{\Gamma_m}<\pi m.
$$
\end{theor}

In the paper \cite{nsv}, the problem of estimating the coefficient of quasiconformal reflection with respect of the boundary of isosceles trapezoids was studied.

Denote by $T=T(\alpha,d)$ the  isosceles trapezoid of unit height with acute angle $\pi\alpha$ and lengths of bases $2c$  and $2d$, $c\le d$,
\begin{equation}\label{d}
d=c+\cot(\pi\alpha).
\end{equation}

In \cite{nsv} some lower and upper estimates of $QR_T$ were established; we will formulate  the results in two theorems below. Introduce the function
$$
g(\lambda)= \frac{\lambda
{\mathcal{K}}'(\lambda)}{{\mathcal{K}}(\lambda)}, \quad  \lambda\in (0,1),
$$
where ${{\mathcal{K}}(\lambda)}$ is the complete elliptic integral of the first kind, defined as
\begin{equation}\label{ek}
{\mathcal{K}}
(\lambda)=\int_0^1\frac{dt}{\sqrt{(1-t^2)(1-\lambda^2t^2)}}\,,
\end{equation}
${\mathcal{K}}'(\lambda)={\mathcal{K}}(\lambda')$, $\lambda'=\sqrt{1-\lambda^2}$.
As it is shown in \cite[lemma~5.2]{nsv},  this function is concave on $(0,1)$ and has a unique maximum at the point $\lambda_0=0.7373921\ldots$ which is a unique root on $(0,1)$ of the equation
$$
(\lambda')^2{\mathcal{K}}(\lambda){\mathcal{K}}'(\lambda)=\pi/2.
$$
Denote
$$
C(\alpha)=\left(\sqrt{1+\tan^2(\pi\alpha)/4}-\tan(\pi\alpha)/2\right)^2,\quad
0< \alpha\le 1/2.
$$

\begin{theor}[\cite{nsv}, corollary~6.4]\label{lower}
Let the ratio of lengths of the bases of a trapezoid $T=T(\alpha,d)$ satisfies $\frac{c}{d}\ge
\lambda_0$, then for the coefficient of quasiconformal reflection with respect to the boundary $L$ of $T=T(\alpha,d)$ we have $$QR_L\ge g(\lambda_0)(1+C(\alpha))d.$$ If
$\frac{c}{d}<\lambda_0$, then $$QR_L\ge g(\lambda)(1+C(\alpha))d,$$
where $\lambda=c/d$.
\end{theor}

\begin{theor}[\cite{nsv}, theorem~7.5]\label{nsvest}
The coefficient of quasiconformal reflection with respect to the boundary  $L=L(\alpha,d)$ of $T=T(\alpha,d)$ satisfies
\begin{equation}\label{tau}
QR_L\le (\sqrt{1+\tau^2}+\tau)^2,
\end{equation}
where
$$
\tau=\max\left\{c+d, \, \,\frac{1-c^2+d^2}{2c}\right\}.
$$
\end{theor}

In \cite{nsv} it was noted that  upper estimate \eqref{nsvest}  is not sharp by order. Actually, for a fixed $\alpha$, taking into account \eqref{d}, we see that the expression in the right-hand side of
\eqref{tau} is equivalent to $Cd^2$, $C=\text{const}$, as $d\to\infty$, though for rectangles (i.e. for $\alpha=1/2$) Theorem~\ref{w} shows that the upper estimate must has the first order of growth with respect to $d$. In \cite{nsv}  it was conjectured that for trapezoids the coefficient must be estimated from above by a linear function in $d$.
In this paper we show that it is the case and prove the following result. \medskip
	
\begin{theor}\label{est} {The coefficient of quasiconformal reflection  $QR_L $ with respect to the boundary $L=L(\alpha,d)$  of an isosceles trapezoid $T=T(\alpha,d)$ satisfies }
\begin{equation}\label{qrl}
QR_L\le \dfrac{\pi\left(\sqrt{(d+c)^2+d^2(d-c)^2}+(d-c)\sqrt{1+d^2}
\right)^4}{8c^2d}\,.
\end{equation}
\end{theor}

We note that for a fixed $\alpha$ the value in the right-hand side of \eqref{qrl} is equivalent to the linear function $C_1(\alpha)d$, as $d\to \infty$,
where
\begin{equation}\label{C1a}
C_1(\alpha)=\dfrac{\pi}{8}\,\left(\sqrt{4+\cot^2(\pi\alpha)}+\cot(\pi\alpha)\right)^4.
\end{equation}

	

\section{Proof of Theorem~\ref{est}}\label{main}

	Consider an isosceles trapezoid $T=T(\alpha,d)$ of height $1$ in the upper half-plane $\mathbb{H}$, symmetric with respect to the imaginary axis,   such that its bigger base is on the real axis and  the angles at this base equal $\pi\alpha$,  $0 < \alpha < 1/2$; as we noted above, the lengths,  $c$ and $d$,   of the bases of $T$ satisfy \eqref{d}.
Our task is to construct a quasiconformal reflection $f$ with respect to $L=\partial T$.

We will construct $f$  as superposition of three mappings, quasiconformal automorphisms of the plane. The first mapping $f_1$ is piecewise smooth, it maps $T$ onto the rectangle $\Pi=[-d,d]\times[0,1]$. Moreover, $f_1$ keeps the symmetry with respect to the imaginary axis. The second mapping $f_2$ is the quasiconformal reflection with respect to the boundary of the rectangle $\Pi$, constructed in \cite{werner}. At last, the third mapping $f_3$ is inverse to~$f_1$.

Therefore, now the main problem is to construct $f_1$. First we define $f_1$ in the right half-plane $\mathbb{H}$ so that it keeps points of the imaginary axis, and then we continue it by symmetry to the whole plane.

Denote by $D$ the rectangular trapezoid that is the right half of $T$.
Then the vertices of $D$ are at the points $0$, $d$, $c+i$, and $i$.
Then we separate the right half-plane into five parts:
$$
G_1=D,	\quad G_2=\{z \in \mathbb{H} : 0 \leq \mbox{\rm Im}\ z \leq 1, z \notin D \},\quad
G_3=\{z \in \mathbb{H} : \mbox{\rm Im}\ z > 1, 0 \leq \mbox{\rm Re}\ z \leq c\},
$$
$$
G_4=\{z \in \mathbb{H} :\mbox{\rm Im}\ z > 1, \mbox{\rm Re}\ z > c\},\quad
G_5=\{z \in \mathbb{H} :\mbox{\rm Im}\ z < 0\}
$$
(Fig.\ref{ris1}).

\begin{figure}[ht] \centering
\includegraphics[width=6.1 in]{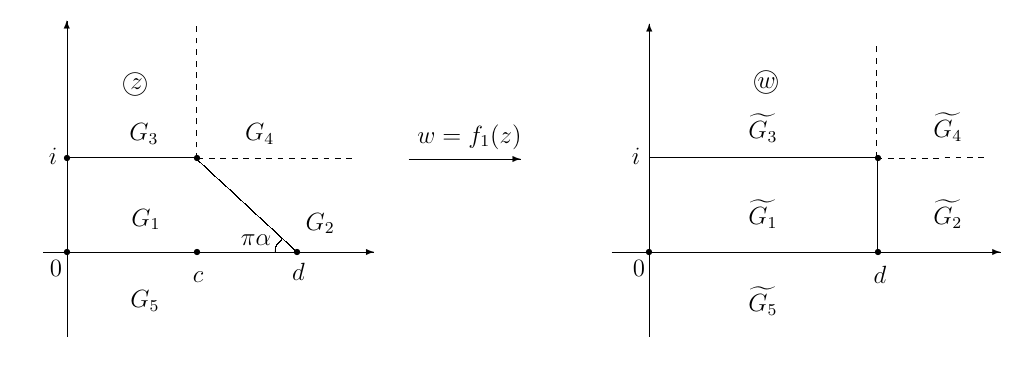}
\caption{The piecewise smooth mapping $f_1$.}\label{ris1}
\end{figure}

Denote $\ell=\frac{d-c}{d}$, $\ell\in(0,1)$.

It is easy to proof the following statement.

\begin{lem}\label{l1}
	{The function }
	\begin{equation}\label{wz}
		f_1(z)\ =\ \left\{
		\begin{array}{cr}
			\dfrac{4z+i\ell(z-\overline{z})^2}{4+i2\ell(z-\overline{z})},&\quad \quad z\in G_1,
			\\[3mm]
			\quad  z - \dfrac{i(d-c)}{2}\,(z-\overline{z}),&\quad z\in G_2,
			\\[3mm]
			\dfrac{(2-\ell)z+\ell \overline{z}}{2(1-\ell)},&\quad z\in G_3,
			\\[3mm]
			z+(d-c),&\quad z\in G_4,
			\\[1mm]
			z,&\quad z\in G_5,
     	\end{array}\right.
	\end{equation}
is a piecewise smooth quasiconformal automorphism of the right half-plane, mapping the trapezoid $G_1$ onto the rectangle $\widetilde{G}_1$ with vertices at $0$, $i$, $d+i$ and $d$.
\end{lem}

We note that constructed function \eqref{wz} maps the domains $G_k$ onto the domains $\widetilde{G}_k$ shown in Fig.1 on the right.

Now we will estimate the dilatation $$D_{f_1}(z)=\frac{|{f_1}_{z}|+|{f_1}_{\bar{z}}|}{|{f_1}_{z}|-|{f_1}_{\bar{z}}|}$$
of the mapping~$f_1$.

\begin{lem}\label{l1}	
	{In the domains $G_k$, the dilatation $D_{f_1}$ of $f_1$ can be estimated as follows:}
	\begin{equation}\label{Df1}
		D_{f_1}(z)\ \leq\  \left\{
		\begin{array}{cr}
			\dfrac{\quad \left(\sqrt{(2-\ell)^2+\ell^2d^2}+\ell\sqrt{1+d^2} \right)^{2}}{4(1-\ell)}\,,&z\in G_1,
			\\[4mm]
			\dfrac{(d-c+\sqrt{(d-c)^2+4})^2}{4}\,,&z\in G_2,
			\\[4mm]
			{1}/{(1-\ell)},&z\in G_3,
			\\[4mm]
			1,&z\in G_4\cup G_5.
		\end{array}
\right.
	\end{equation}
\end{lem}	

\begin{proof}  The estimates in the domains $G_k$, $2\le k\le 5$ is almost evident, therefore, we will only prove the estimate in $G_1$. For $z=x+iy\in G_1$ we have
		\begin{equation*}
			f_1(x,y) = u(x,y)+iv(x,y)
		\end{equation*}
where
\begin{equation*}
u(x,y)=\dfrac{x}{1-\ell y},\quad 	v(x,y)=y.
\end{equation*}
Then
$$
u_{x}=\dfrac{1}{1-\ell y},\quad  v_{x}=0,\quad  u_{y}=\dfrac{\ell x}{(1-\ell y)^2},\quad  v_{y}=1,
$$
consequently,
\begin{equation*}
{f_1}_{z}=\dfrac{1}{2} \left(\dfrac{2-\ell y}{1-\ell y} -i\dfrac{\ell x}{(1-\ell y)^2}\right),
\quad
{f_1}_{\overline{z}}=\dfrac{1}{2} \left(\dfrac{\ell y}{1-\ell y} +i\dfrac{\ell x}{(1-\ell y)^2}\right).
\end{equation*}
		\par Therefore,
\begin{equation}\label{fzz}
\left|\dfrac{{f_1}_{\overline{z}}}{{f_1}_{z}}\right|=\sqrt{\dfrac{(\ell y)^2(1-\ell y)^2+(\ell x)^2}{(2-\ell y)^2(1-\ell y)^2+(\ell x)^2}}\,.
\end{equation}
For a fixed $y\in[0,1]$, the value of $x$ satisfies $0\le x\le d(1-\ell y)$. Since the right-hand side in \eqref{fzz} is an increasing function with respect to $x$ for a fixed $y$, its maximal value is attained at $x=d(1-\ell y)$. Therefore, 
\begin{equation}\label{fzz1}
\left|\dfrac{{f_1}_{\overline{z}}}{{f_1}_{z}}\right|\le\sqrt{\dfrac{(\ell y)^2+(\ell d)^2}{(2-\ell y)^2+(\ell d)^2}}\,.
\end{equation}
At last, because of the fact that the right-hand side of \eqref{fzz} is an increasing function with respect to $y\in[0,1]$, we obtain
\begin{equation}\label{fzz1}
\left|\dfrac{{f_1}_{\overline{z}}}{{f_1}_{z}}\right|\le\sqrt{\dfrac{1+d^2}{(2/\ell -1)^2+d^2}}\,.
\end{equation}
From this we deduce that
$$
D_{f_1}=\dfrac{|{f_1}_{z}|+|{f_1}_{\overline{z}}|}{|{f_1}_{z}|-|{f_1}_{\overline{z}}|} \le \dfrac{\sqrt{(2/\ell -1)^2+d^2}+\sqrt{1+d^2}}{\sqrt{(2/\ell -1)^2+d^2}-\sqrt{1+d^2}}=\dfrac{\left(\sqrt{(2-\ell)^2+\ell^2d^2}+\ell\sqrt{1+d^2} \right)^{2}}{4(1-\ell)}\,
$$
and the needed estimate is proved. \end{proof}

Now we will  estimate the coefficient of quasiconformness of the mapping $f_1$, defined in the right half-plane by \eqref{l1}, comparing the estimates obtained in Lemma~\ref{l1}.

\begin{lem}\label{l3}
 {The mapping $f_1$ is $\widetilde{K}$-quasiconformal where
\begin{equation}\label{wK}
  \widetilde{K} = \dfrac{\left(\sqrt{(2-\ell)^2+\ell^2d^2}+\ell\sqrt{1+d^2} \right)^{2}}{4(1-\ell)}\,.
\end{equation}
  }
\end{lem}

\begin{proof}
To establish the statement of the lemma we need to compare the majorants of $D_{f_1}$ in different domains $G_k$. It is easy that in $G_1$  the majorant from Lemma~\ref{l1} is bigger that those in  $G_3$, $G_4$ and $G_5$. Therefore, it remains to compare the majorants in $G_1$ and $G_2$. Since in $G_2$ we have
$$
\left|\dfrac{{f_1}_{\overline{z}}}{{f_1}_{z}}\right|=\dfrac{d-c}{\sqrt{4+(d-c)^2}}=\dfrac{\ell d}{\sqrt{4+(\ell c)^2}},
$$
and in $G_1$ estimate \eqref{fzz1} holds, it is sufficient to prove the inequality
$$
\dfrac{\ell d}{\sqrt{4+(\ell d)^2}}\le \sqrt{\dfrac{1+d^2}{(2/\ell -1)^2+d^2}}\,=\dfrac{\ell\sqrt{1+d^2}}{\sqrt{(2-\ell)^2+(\ell d)^2}}\,,
$$
which is evident.
\end{proof}

\textit{Proof of Theorem~\ref{est}.} We continue $f_1$, by symmetry, to the whole complex plane. As a result, we obtain a $\widetilde{K}$-quasiconformal mapping of $\mathbb{C}$ onto itself such that the isosceles trapezoid $T=T(\alpha,d)$ is mapped onto the rectangle with base length $2d$ and height~$1$.
After simple transformations, taking into account the equalities
$$
\ell=\frac{\cot (\pi\alpha)}{d}, \quad 1-\ell=\frac{c}{d}\,,\quad 2-\ell=\frac{c+d}{d},
$$
from \eqref{wK} we obtain
$$
  \widetilde{K} = \dfrac{\left(\sqrt{(d+c)^2+d^2\cot^2(\pi\alpha)}+\cot(\pi\alpha)\sqrt{1+d^2}
  \right)^2}{4cd}\,.
$$

Let $f_2$ be the quasiconformal reflection with respect to the boundary of the indicated rectangle built in \cite{werner}. By Theorem~\ref{w}, it is $\widetilde{K}_1$-quasiconformal with coefficient $\widetilde{K}_1=2\pi d$.
Since the superposition of a $ K_{1}$-quasiconformal and $K_{2}$-quasiconformal  mappings is $ K_{1}K_{2} $-quasiconformal, we wee that $f_1^{-1}\circ f_2\circ f_1$ is a $(\widetilde{K})^2\widetilde{K}_1$-quasiconformal mapping, moreover, it is a reflection with respect to $L$. Thus, Theorem~\ref{est} is proved.
\hfill $\square$

Using the reasoning applied in the proof of Theorem~\ref{est}, we can prove the following fact.

\begin{theor}\label{paral}
Let $\Pi=\Pi(\alpha,a)$ be a parallelogram with acute angle $\pi \alpha$, two parallel sides of which have length $a$, $a>\cot(\pi\alpha)$. 
The coefficient of quasiconformal reflection with respect to the boundary $L=L(\alpha,a)$ of $\Pi$ satisfies
$$
QR_L \le \frac{\pi\left(\sqrt{4a^2+(a+\cot(\pi\alpha))^2\cot^2(\pi\alpha)}+\cot(\pi\alpha)\sqrt{4+(a+\cot(\pi\alpha))^2}\right)^4}{16(a+\cot(\pi\alpha))(a-\cot(\pi\alpha))^2}\,.
$$
\end{theor}

The proof of Theorem~\ref{paral} is also based on the idea of reduction of the problem to reflection with respect to the boundary of a rectangle.  We can consider that the parallelogram is symmetric with respect to the point $i/2$ and the sides of length are horizontal. The imaginary  axis separates $\Pi$ into two rectangular trapezoids. Then, when constructing the mapping $f_1$ in the right half-plane, we use the same formulas, as above, and then continue the mapping to the left half-plane using  the central symmetry with respect to the point $i/2$, not mirror one.
\medskip

In conclusion, we compare the results given in Theorems~\ref{nsvest} and \ref{est} with a fixed angle $\pi\alpha$, as $d\to\infty$. The majorant $(\sqrt{1+\tau^2}+\tau)^2$ from Theorem~\ref{nsvest} is equivalent to the value $4\tau^2\sim 4(c+d)^2\sim16 d^2$, $d\to\infty$, as long as the majorant from Theorem~\ref{est} is equivalent to $C_1(\alpha)d$ where $C_1(\alpha)$ is defined in \eqref{C1a}. This shows that is $\pi\alpha$ is close to $\pi/2$, then the estimate from Theorem~\ref{est} gives better results. At the same time, of $\pi\alpha$ is close zero, then for small values of $d$ the esitmates from Theorem~\ref{nsvest} may be better.

In Fig.\ref{ris2} we give the graphs of majorants from Theorems~\ref{nsvest} and \ref{est} as functions of the smaller side length $c$\footnote{For convenience, here we use the  smaller base of
length $c$, since the set of all admissible values for $c$ coincide with $\mathbb{R}_+$; for a fixed $\alpha$,  the lengths $d$ and $c$ differ by a constant, see \eqref{d}.} of the trapezoid for $\alpha=0.45$ and $\alpha=0.3$, illustrating our conclusions.

\begin{figure}[h]\label{ris2}
\begin{minipage}[h]{0.49\linewidth}
\center{\includegraphics[width=0.8\linewidth]{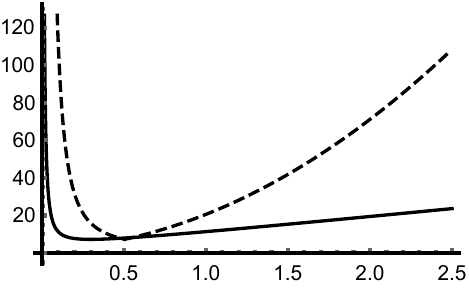} \\ a)}
\end{minipage}
\hfill
\begin{minipage}[h]{0.49\linewidth}
\center{\includegraphics[width=0.8\linewidth]{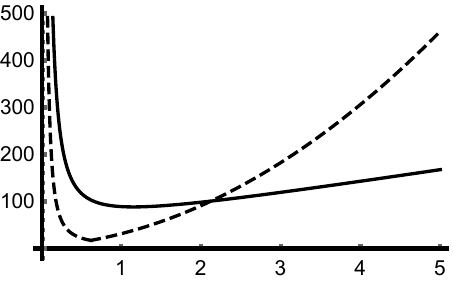} \\ b)}
\end{minipage}
            \caption{The graphs of the majorants for $QR_L$ from Theorems~\ref{nsvest} (dashed lines) and \ref{est} (solid lines) in the cases a)~$\alpha=0.45$  and b)~$\alpha=0.3$.}
\label{ris:image1}
\end{figure}

\medskip

The work of the third author was performed under the development program of the Volga Region Mathematical Center (agreement no. 075-02-2024-1438).

\end{document}